\documentclass{amsart}[11pt,letter]
\usepackage[margin=2.54cm]{geometry}
\usepackage{amsmath,amssymb,amsthm,url,tikz,chngcntr}
\usetikzlibrary{positioning}
\counterwithin{table}{section}
\counterwithin{figure}{section}
\usepackage{hyperref}

\newtheorem{theorem}{Theorem}[section]
\newtheorem{lemma}[theorem]{Lemma}

\numberwithin{equation}{section}

\newtheorem{notation}[theorem]{Notation}

\theoremstyle{definition}
\newtheorem{definition}[theorem]{Definition}

\renewcommand{\wr}{\mathop{\rm wr}}

\newcommand{\Fix}{\mathrm{Fix}}
\newcommand{\m}{{\mathrm{mindeg}}}
\newcommand{\fix}{\mathrm{fix}\,}

\newcommand{\Aut}{\mathrm{Aut}}

\newcommand{\supp}{\mathrm{supp}}
\newcommand{\Alt}{\mathrm{Alt}}
\newcommand{\Sym}{\mathrm{Sym}}

\renewcommand{\O}{\mathop{\mathrm{O}}\nolimits}
\newcommand{\PSU}{\operatorname{\mathrm{PSU}}}
\newcommand{\PSp}{\mathop{\mathrm{PSp}}\nolimits}
\newcommand{\POmega}{\mathop{\mathrm{P}\Omega}\nolimits}
\newcommand{\Sp}{\operatorname{\mathrm{Sp}}}
\newcommand{\PSL}{\mathrm{PSL}}

\newcommand{\fpr}{\mathrm{fpr}}

\def\cent#1#2{{\bf C}_{{#1}}({{#2}})}

\begin{document}
\title[Minimal degree of permutation groups]{On minimal degree of transitive permutation groups with stabiliser being a $2$-group}

\author{Primo\v{z} Poto\v{c}nik}
\address{Faculty of Mathematics and Physics, University of Ljubljana, Jadranska 21, SI-1000 Ljubljana, Slovenia;\\
 also affiliated with 
Institute of Mathematics, Physics, and
  Mechanics, Jadranska 19, SI-1000 Ljubljana, Slovenia}
 \email{primoz.potocnik@fmf.uni-lj.si}

\author{Pablo Spiga}
\address{Dipartimento di Matematica e Applicazioni, University of Milano-Bicocca, Via Cozzi 55, 20125 Milano, Italy} 
\email{pablo.spiga@unimib.it}

\begin{abstract}
The minimal degree of a permutation group $G$ is defined as the minimal number of non-fixed points of a non-trivial element of $G$. In this paper we show that if $G$ is a transitive permutation group of degree $n$ having no non-trivial normal $2$-subgroups such that the stabiliser of a point is a $2$-group, then the minimal degree of $G$ is at least $\frac{2}{3}n$. The proof depends on the classification of finite simple groups.
\end{abstract}

\thanks{The first-named author gratefully acknowledges the support of the Slovenian Research Agency ARRS, core funding programme P1-0294 and research project J1-1691.}

\maketitle

\section{Introduction}\label{sec:intro}
%Throughout this paper, all groups considered will be finite. 
Given a finite  group acting on a finite set $\Omega$ and $g\in G$, we let 
$\Fix_\Omega(g):=\{\omega \in \Omega \mid \omega^g=\omega\}$
 denote the set of {\em fixed points} of $g$ and we let $\supp_\Omega(g):=\Omega \setminus \Fix_\Omega(g)$. The parameter
 $$
 \m_\Omega(G) := \min_{g\in G\setminus\{1\}} |\supp_\Omega(g)|
 $$
is then called the {\em minimal degree} of the permutation group $G$.
Minimal degree (and a related parameter $\Fix_\Omega(G) := 
\max_{g\in G\setminus\{1\}} |\Fix_\Omega(g)| =
|\Omega| - \m_\Omega(G)$, sometimes called the {\em fixity} of $G$) has been the focus of intense study by several authors with most work concentrating on proving upper and lower bounds on the minimal degree of primitive permutation groups and permutation actions of classical groups (see \cite{Bab,Tim1,GluMag,La}, to name a few). Far less papers deal with other permutation groups (see \cite{KPS,NewOBrSha}).

The aim of this paper is to consider a lower bound on the minimal degree of a transitive permutation group $G$ whose point stabiliser is a $2$-group and such that ${\bf O}_2(G)=1$ (where by ${\bf O}_2(G)$ we denote the largest normal $2$-subgroup of $G$). In particular, we prove the following:

\begin{theorem}
\label{proposition:7}
Let $G$ be a transitive permutation group on a set $\Omega$ with ${\bf O}_2(G)=1$ such that the point stabiliser $G_\omega$ is a $2$-group. 
Then $\m_\Omega(G) \ge  2|\Omega|/3$. 
\end{theorem}

Our motivation to prove this result stems from the
investigation of the minimal degree of permutation groups
appearing as arc-transitive automorphism groups of finite   connected (possibly directed) graphs,
or equivalently, permutation groups admitting a   connected suborbit,
where by a   connected suborbit of a transitive permutation group $G\le \Sym(\Omega)$ we mean an orbit $\Sigma$ of the point-stabiliser $G_\omega$ acting on $\Omega\setminus\{\omega\}$, such that the directed graph with the vertex-set $\Omega$
and the edge-set  $\{ (\omega^g,\delta^g) \mid \delta\in \Sigma, g\in G\}$ is   connected.
 Note that a transitive permutation group $G\le\Sym(\Omega)$ is primitive if and only if all of its non-trivial suborbits are   connected.
In view of the numerous results regarding the minimal degree of permutation groups, it is thus natural to relax the condition of primitivity to that of an
existence of a single   connected suborbit. In fact, based on the available computational data \cite{conder,condercensus,census1,census2}, one can see that under some additional assumptions on the length and the self-pairity of the   connected suborbit (and modulo some well-defined exceptions) lower bounds 
on the minimal degree holding for primitive permutation groups seem to extend to this more general situation.
 Theorem~\ref{proposition:7}  represents a crucial step in a recent result~\cite{PotSpi}
stating that apart from a well-understood infinite family and a finite set of examples,
every transitive permutation group $G\le\Sym(\Omega)$ admitting 
a   connected suborbit of length at most $2$ or a
self-paired
   connected suborbit of length at most $4$ satisfies $\m_\Omega(G) \ge  2|\Omega|/3$ (a suborbit $\Sigma$ is {\em self-paired} if for every edge $(\omega',\omega'')$ of the corresponding directed graph also the pair $(\omega'',\omega')$ is an edge of that directed graph).

We would also like to point out that Guralnick and Magaard~\cite{GurMag} (building on an earlier work of Lawther, Liebeck and Saxl~\cite{La,LS}) have proved that, except for an explicit list of exceptions, the minimal degree of every primitive permutation group of degree $n$ is at least $n/2$. Our proof of Theorem~\ref{proposition:7} quickly reduces to the case that the group $G$ under consideration is quasiprimitive. In this sense our result can be thought of as an attempt to extend the result of Guralnick and Magaard to the case of quasiprimitive groups.

We use a fairly standard notation.
Given a set $\Omega$, we denote by $\Sym(\Omega)$ and $\Alt(\Omega)$ the symmetric and the alternating group on $\Omega$. When the domain $\Omega$ is irrelevant or clear from the context, we write $\Sym(n)$ and $\Alt(n)$ for the symmetric and alternating group of degree $n$.
Given a permutation $g\in \Sym(\Omega)$, we write $\fpr_\Omega(g)$ for the {\em fixed-point-ratio} of $g$, that is 
$$
\fpr_\Omega(g):=\frac{|\Fix_\Omega(g)|}{|\Omega|}.
$$

%Given $n\in\mathbb{N}\setminus\{0\}$, we denote by $\D_n$ the dihedral group of order $2n$ and we view $\D_n$ as a permutation group of degree $n$; similarly, we denote by $\C_n$ the cyclic group of order $n$. Similarly, we denote by $\mathbb{Z}_n$ the integers modulo $n$.

A subgroup $G$ of $\Sym(\Omega)$ is said to be {\em semiregular} if the identity is the only element of $G$ fixing some point of $\Omega$.
 Let $G$ be a group and let $H$ be a subgroup of $G$, we denote by $H\backslash G$ the {\em set of right cosets} of $H$ in $G$. Recall that $G$ acts transitively on $H\backslash G$ by right multiplication. The normaliser of $H$ in $G$ is denoted by ${\bf N}_G(H)$.

\section{Proof of Theorem~$\ref{proposition:7}$}

We begin by two useful lemmas and then proceed to the proof of Theorem~$\ref{proposition:7}$.

\begin{lemma}\label{lemma:4}Let $X$ be a group acting transitively on $\Omega$, let $\Sigma$ be a system of imprimitivity for $X$ in its action on $\Omega$ and let $x\in X$. Then $\fpr_{\Omega}(x)\le \fpr_{\Sigma}(x)$.
\end{lemma}
\begin{proof}
This is clear because, if $x$ fixes a block $B\in \Sigma$, then $\Fix_{B}(x)\subseteq B$ and hence $|\Fix_{B}(x)|\le |B|$.
\end{proof}

\begin{lemma}\label{eq:22}
Let $X$ be a group acting on a set $\Omega$, let $Y$ be a normal subgroup of $X$, let $x\in X$ and let $\omega\in \Omega$. Then
\begin{equation}\label{eq:2}
\fpr_{\omega^Y}(x)=\frac{|x^Y\cap X_\omega|}{|x^Y|},
\end{equation}
where $x^Y:=\{x^y\mid y\in Y\}$ is the $Y$-conjugacy class of the element $x$.
\end{lemma}
\begin{proof}
This equality is classic, see for instance~\cite{LS}. Here we present a short proof: consider the bipartite graph with one side of the bipartition labeled by the elements of  $x^Y$ and the other side of the bipartition labeled by the elements of $\omega^Y$. Declare $x'\in x^Y$ adjacent to $\omega'$ if $x'$ fixes $\omega'$. Clearly, $x'\in x^Y$ has valency $|\Fix_{\omega^Y}(x')|=|\Fix_{\omega^Y}(x)|$ and $\omega'\in \omega^Y$ has valency $|x^Y\cap X_{\omega'}|=|x^Y\cap X_{\omega}|$. 
 Thus $|x^Y||\Fix_{\omega^Y}(x)|=|x^Y\cap X_\omega||\omega^Y|$.
\end{proof}

The rest of the section is devoted to the proof Theorem~\ref{proposition:7}.
Let $G$ be a transitive permutation group acting on a set $\Omega$, $|\Omega|=n$, with ${\bf O}_2(G)=1$ such that the point stabiliser $G_\omega$ is a $2$-group. 
We need prove that $\fix_\Omega(G) \le n/3$. 
Our proof is by induction on $|\Omega|+|G|$.

 Let $Q$ be a Sylow $2$-subgroup of $G$ with $G_\omega\le Q$ and let $B:=\omega^Q$. Since $G_\omega \le Q$,  the set $\Sigma:=\{B^x\mid x\in G\}$ is a $G$-invariant partition of $\Sigma$ upon which $G$ acts with the stabiliser of the element $B\in \Sigma$ being $Q$.
 As ${\bf O}_2(G)=1$ and $Q$ is a $2$-group, the action of $G$ on $\Sigma$ is faithful. If $|\Sigma|<|\Omega|$, then by induction $\fpr_\Sigma(g)\le 1/3$, for every $g\in G\setminus\{1\}$. Therefore, from Lemma~\ref{lemma:4}, we have $\fpr_\Omega(g)\le 1/3$, for every $g\in G\setminus\{1\}$. We may thus suppose that $Q=G_\omega$, that is, 
\begin{equation}\label{sylow}
G_\omega \textrm{ is a Sylow }2\textrm{-subgroup of }G.
\end{equation}

Let $N$ be a minimal normal subgroup of $G$ and let $K$ be the kernel of the action of $G$  on the set $\Sigma$ of $N$-orbits. Suppose that $N$ is not transitive on $\Omega$. In particular, $G/K$ is a non-identity transitive permutation group on $\Sigma$ of odd degree greater than $1$; moreover, given $B\in \Sigma$ and $\omega\in B$, the setwise stabilizer of $B$ in $G$ is $G_\omega K$ and $G_\omega K/K$ is a Sylow $2$-subgroup of $G/K$. By induction, if $g\in G\setminus K$, then $\fpr_\Sigma(gK)\le 1/3$ and hence $\fpr_{\Omega}(g)\le \fpr_{\Sigma}(gK)\le 1/3$. Suppose that $g\in K\setminus\{1\}$. Since $K\unlhd G$ and $G_\omega$ is a Sylow $2$-subgroup of $G$, we deduce that $K_\omega$ is a Sylow $2$-subgroup of $K$. In particular, from the Frattini argument, we have 
$$G=K{\bf N}_G(K_\omega).$$ 
%\color{magenta}
%Therefore, ${\bf N}_G(K_\omega)$ acts transitively on $\Sigma$, implying that
%$K_\omega$ fixes at least one vertex in every $K$-orbit in $\Sigma$.
%The core of $K_\omega$ in $K$ this acts trivially on every $B\in \Sigma$, implying that $K_\omega$ is core-free in $K$.
%\color{black}
%\color{brown}
This implies that the core of $K_\omega$ in $K$ equals the core of $K_\omega$ in $G$, which is trivial,
implying that $K$ acts faithfully on $\omega^K$.
%Since $G_\omega$ is core-free in $G$, from this equality we deduce that $G_\omega\cap K$  is core-free in $K$. 
%\color{black}
Since this argument does not depend upon $\omega\in \Omega$, $K$ acts faithfully on each  element of $\Sigma$. In particular, for every $B\in \Sigma$, the restriction of $g$ to $B$ is a non-identity permutation and hence, by induction, $\fpr_B(g)\le 1/3$. Since this argument does not depend upon $B\in\Sigma$, we have $\fpr_\Omega(g)\le 1/3$.
 
It remains to deal with the case that every minimal normal subgroup of $G$ is transitive on $\Omega$, that is, $G$  is quasiprimitive.

The class of quasiprimitive permutation groups may be described (see~\cite{11_1}) in
a fashion very similar to the description given by the O'Nan-Scott Theorem for
primitive permutation groups. In~\cite{13_1} this description is refined and eight types
of quasiprimitive groups are defined, namely HA, HS, HC, SD, CD, TW, PA and
AS, such that every quasiprimitive group belongs to exactly one of these types. As our group $G$ has odd degree, it is readily seen, using the terminology in~\cite{11_1,13_1}, that it is of HA, AS or PA type, that is, Holomorphic Abelian, Almost Simple or Product Action. We refer the reader to~\cite{11_1,13_1} for more informations on the structure of groups of HA, AS or PA type, or to~\cite{PrSn} for an extensive treatment of permutation groups and Cartesian decompositions.

Assume that $G$ has O'Nan-Scott type HA; let $V$ be the socle of $G$, let $g\in G\setminus\{1\}$ with $\Fix_\Omega(g)\ne\emptyset$ and let $\omega\in \Fix_\Omega(g)$. As ${\bf O}_2(G)=1$, $V$ is an elementary abelian $p$-group, for some prime $p>2$. Then $G=V\rtimes G_\omega$ and the action of $G$ on $\Omega$ is permutation equivalent to the natural holomorph action of $G$ on $V$, with $V$ acting by right multiplication and with $G_\omega$ acting by conjugation. Using this identification, we have $$\fpr_\Omega(g)=\frac{1}{|V:\cent V g|}\le \frac{1}{p}\le \frac{1}{3}.$$

Assume that $G$ has O'Nan-Scott type PA.
Following~\cite{11_1,13_1}, we  set some notation and state some facts
regarding the groups of the O'Nan-Scott type PA.
Let $g\in G\setminus\{1\}$ with $\Fix_\Omega(g)\ne\emptyset$, let $\omega\in \Fix_\Omega(g)$ and $N$ be the socle of $G$. Then there exists a finite non-abelian simple group $T$ such that $N=T_1\times T_2\times\cdots \times T_\ell$ for some $\ell \ge 2$ with $T_i\cong T$ for each $i$. The group $G$ can then be identified with a subgroup of $\Aut(T)\wr \Sym(\ell)$. Let $R$ be a Sylow $2$-subgroup of $T$ and let $\Delta$ be the set $R\backslash T$ of right cosets of $R$ in $T$.
As $G_\omega$ is a Sylow $2$-subgroup of $G$, we have 
\begin{equation*}
G_\omega\cap N=R_1\times \cdots \times R_\ell\cong R^\ell,
\end{equation*}
where $R_i$ is a Sylow $2$-subgroup of $T_i$ for each $i$. From~\cite{11_1,13_1}, the action of  $G$ on $\Omega$ is permutation isomorphic to the natural Cartesian product action of $G$ on $\Delta^\ell$. By identifying $\Omega$ with $\Delta^\ell$, we have $G\le W$ with $W:=\Aut(T)\wr \Sym(\ell)$, where $W$ acts on $\Omega$ with the Cartesian product action. In particular, we may write the elements $x\in W$ in the form 
$$x=(a_1,a_2,\ldots,a_\ell)\sigma,$$ for some $a_1,a_2,\ldots,a_\ell\in \Aut(T)$ and $\sigma\in \Sym(\ell)$. Recall that, if $(\delta_1,\delta_2,\ldots,\delta_\ell)\in \Delta^\ell=\Omega$, then
\begin{align*}
(\delta_1,\delta_2,\ldots,\delta_\ell)^x&=
(\delta_1,\delta_2,\ldots,\delta_\ell)^{(a_1,a_2,\ldots,a_\ell)\sigma}
=(\delta_1^{a_1},\delta_2^{a_2},\ldots,\delta_\ell^{a_\ell})^\sigma
=(\delta_{1\sigma^{-1}}^{a_{1\sigma^{-1}}},\delta_{2\sigma^{-1}}^{a_{2\sigma^{-1}}},\ldots,
\delta_{\ell\sigma^{-1}}^{a_{\ell\sigma^{-1}}}).
\end{align*}
We write $g=(a_1,a_2,\ldots,a_\ell)\sigma$.

Suppose $\sigma\ne 1$. Using the explicit description of $g$ and of the action of $g$ on $\Omega$, with a computation we obtain
$$\fpr_{\Omega}(g)\le \frac{1}{|\Delta|}.$$
As $|\Delta|=|T:R|$ and $T$ is a non-abelian simple group, we have $|\Delta|\ge 3$ and hence $\fpr_\Omega(g)\le \/3$.

Suppose $\sigma= 1$. As $g\ne 1$, without loss of generality, we may assume $a_1\ne 1$. Then $$\Fix_{\Omega}(g)=\Fix_{\Delta^\ell}(g)\subseteq \Fix_{\Delta}(a_1)\times \Delta^{\ell-1}.$$ Since $\ell\ge 2$, we have $|\Delta|<|\Omega|$ and hence, by induction,  $\fpr_{\Delta}(a_1)\le 1/3$ and $\fpr_\Omega(g)\le 1/3$.

It remains to deal with the case that $G$ is an almost simple group. Let $T$ be the socle of $G$.  We now divide the proof in four parts, depending on whether $T$ is a sporadic, an alternating, an exceptional group of Lie type or a classical group.
%%%%%%%%%%%%%%%%%%%%%%%%%%%%%%%%
%%%%%%%%%%%%%%%%%%%%%%%%%%%%%%%%
%%%%%%%%%%%%%%%%%%%%%%%%%%%%%%%%

%%%%%%%%%%%%%%%%%%%%%%%%%%%%%%%%
%%%%%%%%%%%%%%%%%%%%%%%%%%%%%%%%
%%%%%%%%%%%%%%%%%%%%%%%%%%%%%%%%
\subsection{Sporadic groups}

Here we consider the situation where the socle $T$ is one of the sporadic finite simple groups.
The proof is entirely computational and uses the astonishing package  ``The GAP character Table Library"~\cite{AtlasRep} implemented in the computer algebra system \texttt{GAP}~\cite{GAP4}.  For sporadic groups, the proof of Theorem~\ref{proposition:7} follows immediately from Lemma~\ref{lemma:4} and from Lemma~\ref{l:SP}.

\begin{lemma}\label{l:SP}
Let $G$ be an almost simple primitive group on $\Omega$ with socle a sporadic simple group. Then
$$\fpr_{\Omega}(g)\le \frac{1}{3},$$
for every $g\in G\setminus\{1\}$, except when $G=\Aut(M_{22})$ in its primitive action of degree $22$ where the maximum fixed point ratio is $4/11$. Moreover, for each $2$-subgroup $Q$ of $\Aut(M_{22})$ and for each $x\in \Aut(M_{22})$ with $g\ne 1$, in the action of $\Aut(M_{22})$ on the set $Q\setminus \Aut(M_{22})$ of cosets of $Q$ in $\Aut(M_{22})$, we have $\fpr_{Q\backslash \Aut(M_{22})}(g)\le 3/55<1/3$.
\end{lemma}
\begin{proof}
 Apart from 
\begin{itemize}
\item the Monster and
\item the action of the Baby Monster on the cosets of a maximal subgroup of type $(2^2\times F_4(2)):2$,
\end{itemize} each permutation character of each primitive permutation representation of an almost simple group with socle a sporadic simple group is available in \texttt{GAP} via the package ``The GAP character Table Library". Therefore, except for the two cases mentioned above,  we can quickly and easily use \texttt{GAP} to test the veracity of  the lemma.  The permutation character of the Baby Monster $G$ on the cosets of a maximal subgroup $M$ of type $(2^2\times F_4(2)):2$ is missing from the \texttt{GAP} library because the conjugacy fusion of some of the elements of $M$ in $G$ remains a mystery: this information is vital for computing the permutation character.

In the rest of the proof we use Lemma~\ref{eq:22}.
Suppose then $G$ is the Baby Monster. Let $G_\omega$ be the stabilizer in $G$ of the point $\omega\in \Omega$ and suppose $G_\omega\cong (2^2\times F_4(2)):2$. Using the ATLAS notation, with a direct computation we see that
$$\fpr_\Omega(g)=\frac{|g^G\cap G_\omega|}{|g^G|}\le \frac{|G_\omega|}{|g^G|}\le \frac{1}{3},$$
unless $g$ is in the conjugacy class $1A$, $2A$ or $2B$. In particular, the lemma is proved also in this case except when $g$ is in the conjugacy class $2A$ or $2B$. Let us denote by $t$ the number of solutions to the equation $x^2=1$ in $G_\omega$. From \cite[(4.6)]{Isaacs}, we see that
$$t=\sum_{\chi\in \mathrm{Irr}^+(G_\omega)}\chi(1)-\sum_{\chi\in \mathrm{Irr}^-(G_\omega)}\chi(1),$$
where $\mathrm{Irr}^+(G_\omega)$ and $\mathrm{Irr}^-(G_\omega)$ are the sets of the irreducible complex characters of $G_\omega$ of orthogonal and of symplectic type. As the character table of $G_\omega$ is available in GAP, we can compute $t$ with this formula and we obtain that $t=1605784576$. Therefore, when $g$ is in the conjugacy class $2A$ and $2B$, we can refine the previous bound and we obtain
$$\fpr_{\Omega}(g)=\frac{|g^G\cap G_\omega|}{|g^G|}\le \frac{t}{|g^G|}\le \frac{1}{3}.$$

\smallskip

For the rest of this proof we may assume that $G$ is the Monster group. Let $\omega\in \Omega$. From~\cite[Section~3.6]{wilsonArXiv}, we see that the classification of the maximal subgroups of the Monster  is complete except for a few small open cases. Suppose first that $G_\omega$ is isomorphic to the double cover of the Baby monster. Let $\pi$ be the permutation character of $G$ on $\Omega$. It was proved by Breuer and Lux~\cite[page~2309]{BreuerLux} that $$\pi=\chi_1+\chi_2+\chi_4+\chi_5+\chi_9+\chi_{14}+\chi_{21}+\chi_{34}+\chi_{35}$$ (this was also proved independently in \cite{MS}).
With this character we can check that no non-identity element of $G$ fixes more than $1/3$ of the points. (Accidentally, as far as we are aware, $\pi$ is the only permutation character of the Monster that has been computed.) For the rest of the proof, we may assume that $G_\omega$ is not conjugate to the double cover of the Baby monster. From~\cite[Section~3.6]{wilsonArXiv}, it follows that $|G_\omega|\le |2^{1+24}.Co_1|$. Now, it is an easy computation to check that
$$\fpr_\Omega(g)\le \frac{|g^G\cap G_\omega|}{|g^G|}\le \frac{|G_\omega|}{|g^G|}\le\frac{|2^{1+24}.Co_1|}{|g^G|}<\frac{1}{3},$$
except when $g$ is in the conjugacy class $1A$ and $2A$.  Therefore, for the rest of the proof we may assume that $g$ is in the conjugacy class $2A$. From \cite[Lemma 2.7]{LS}, we have
$$\fpr_{\Omega}(g)\le \max\left\{\frac{1+|\chi(g)|}{1+\chi(1)}\mid \chi\in \mathrm{Irr}(G),\chi(1)\ne 1\right\}.$$
This quantity can be easily computed and it is less than $1/3$.
\end{proof}

This allows us to finish the proof of Theorem~\ref{proposition:7} when $T$ is a sporadic simple group:
Let $\Sigma$ be a maximal system of imprimitivity for $G$ acting on $\Omega$. 
Except when $G=\Aut(M_{22})$, the proof follows applying Lemmas~\ref{lemma:4} and~\ref{l:SP} to $\Sigma$. When $G=\Aut(M_{22})$, the proof follows directly from Lemma~\ref{l:SP}. 

%%%%%%%%%%%%%%%%%%%%%%%%%%%%%%%%
%%%%%%%%%%%%%%%%%%%%%%%%%%%%%%%%
%%%%%%%%%%%%%%%%%%%%%%%%%%%%%%%%
\subsection{Exceptional groups of Lie type}

Suppose now that the socle $T$ of $G$ is an exceptional groups of Lie type.
Exactly as for the sporadic groups, we prove a much stronger statement than needed for the proof of Theorem~\ref{proposition:7}.

\begin{lemma}\label{l:EX}
Let $G$ be an almost simple primitive group on $\Omega$ with socle an exceptional simple group of Lie type. Then
$\fpr_{\Omega}(g)\le \frac{1}{3}$
for every $g\in G\setminus\{1\}$.
\end{lemma}
\begin{proof}
 For exceptional groups of Lie type Lawther,
Liebeck and Seitz \cite[Theorem 1]{La} have obtained useful and explicit upper bounds
on $\fpr_\Omega (g)$. From these bounds, it readily follows that $\fpr_\Omega(g)\le 1/3$, for every element $g\in G\setminus\{1\}$, except when the socle of $G$ is $G_2(2)'$ and $^2G_2(3)'$. Now, $G_2(2)'\cong \mathrm{PSU}_3(3)$ and  $^2G_2(3)'\cong \mathrm{PSL}_2(8)$; these cases can be analyzed with a help of a computer. The maximum fixed point ratio for $G$ when the socle is $\mathrm{PSL}_2(8)$ is $1/3$ (arising from the natural action of $\mathrm{P}\Gamma\mathrm{L}_2(8)$ on the nine points of the projective line). The same holds for groups having socle $\mathrm{PSU}_3(3)$, the maximum $1/3$ is achieved on the primitive action of degree  $36$. 
\end{proof}

The proof of Theorem~\ref{proposition:7} in the case when $T$ is an exceptional simple group of Lie type now follows by applying
Lemmas~\ref{lemma:4} and~\ref{l:EX} to
 a maximal system of imprimitivity $\Sigma$ of the action of $G$  on $\Omega$. 

%%%%%%%%%%%%%%%%%%%%%%%%%%%%%%%%
%%%%%%%%%%%%%%%%%%%%%%%%%%%%%%%%
%%%%%%%%%%%%%%%%%%%%%%%%%%%%%%%%
\subsection{Alternating groups}

Suppose now that  $T$ (as an abstract group) is isomorphic to the alternating group $\Alt(n)$, for some $n\in\mathbb{N}$ with $n\ge 5$. For this proof, 
we argue  by induction on $n$. We first consider the case that $n\in \{5,6,7,8\}$, this will avoid some  detour  in our arguments. The result in this case follows with a computation with the invaluable computer algebra system magma~\cite{magma}. From now on we may assume that $n\ge 9$ and hence, in particular, $G=\Alt(n)$ or $G=\Sym(n)$. Since the action of $\Alt(n)$ on the cosets of one of its Sylow $2$-subgroups extends to an action of $\Sym(n)$, we may assume that  
 $G=\Sym(n)$.

\smallskip
\noindent\textsc{Case $n$ odd.}
\smallskip

\noindent Let $H$ be a subgroup of $G$ with $H\cong \Sym(n-1)$ and  $G_\omega\le H$, for some $\omega\in \Omega$. (Observe that this is possible because $n$ is odd.) Let $\Sigma$ be the system of imprimitivity determined by the overgroup $H$ of $G_\omega$. Clearly,
$$\Fix_{\Omega}(g)=\bigcup_{B\in\Fix_{\Sigma}(g)}\Fix_{B}(g).$$
As $n-1\ge 5$ and as $H$ acts faithfully on each of its orbits on $\Omega$, we deduce by induction that $|\Fix_{B}(g)|\le |B|/3$, for each $B\in \Fix_{\Sigma}(g)$. Therefore, $\fpr_{\Omega}(g)\le 1/3$.

\smallskip
\noindent\textsc{Case $n$ even.}
\smallskip

\noindent Let $H$ be a subgroup of $G$ isomorphic to the imprimitive wreath product $\Sym(n/2)\wr \Sym(2)$ and with $G_\omega\le H$, for some $\omega\in \Omega$. (Observe that this is possible because $n$ is even.) As above, $$\Fix_{\Omega}(g)=\bigcup_{B\in\Fix_{\Sigma}(g)}\Fix_{B}(g).$$ Let $B\in \Fix_\Sigma(g)$ and, when $\Fix_B(g)\ne\emptyset$, let $\omega'\in \Fix_B(g)$. Now, $G_\omega\le G_B\cong H$ and hence, without loss of generality, we may suppose that $\omega'=\omega$ and $G_B=H$. In what follows we aim to estimate $\Fix_B(g)$.

Now, $G_\omega=P\wr \Sym(2)$, where $P$ is a Sylow $2$-subgroup of $\Sym(n/2)$. The action of $H=\Sym(n/2)\wr\Sym(2)$ on the cosets of $P\wr\Sym(2)=G_\omega$ (that is, the action of $G$ on $\Omega$) is permutation equivalent to the natural product action of $\Sym(n/2)\wr\Sym(2)$ on the Cartesian product $\Delta^2$, where $\Delta=P\backslash \Sym(n/2)$. (This is clear from the structure of $H$ and $G_\omega$, and we refer to~\cite{PrSn} for more details on Cartesian decompositions of permutation groups.) Now, we can write the element $g$ in the form $(h_1,h_2)$ or in the form $(h_1,h_2)(1\,2)$, where $h_1,h_2 \in P$ and $(1\,2)$ is the element swapping the two factors of $\Delta^2$, that is, $(\delta_1,\delta_2)^{(1\,2)}=(\delta_2,\delta_1)$ for every $(\delta_1,\delta_2)\in \Delta$. A permutation of the second kind fixes at most $|\Delta|$ points and hence $$\fpr_{\Omega}(g)=\fpr_{\Delta^2}(g)\le \frac{1}{|\Delta|}=\frac{1}{|\Sym(n/2):P|}\le \frac{1}{3}.$$ A permutation of the first kind has the property that either $h_1\ne 1 $ or $h_2\ne 1$. Since $n/2\ge 5$, we may apply induction and say that the non-identity element $h_1$ or $h_2$ fixes at most $1/3$ of its domain and hence so does $g$.

%%%%%%%%%%%%%%%%%%%%%%%%%%%%%%%%
%%%%%%%%%%%%%%%%%%%%%%%%%%%%%%%%
%%%%%%%%%%%%%%%%%%%%%%%%%%%%%%%%
\subsection{Classical groups}
Suppose finally that the socle $T$ of $G$ is a simple classical group defined over the finite field of size $q$ and recall that we may assume that $T\le G\le \Aut(T)$.

\begin{notation}\label{generalnotation}{\rm
% We let $G$ denote a finite almost simple classical group defined over the finite field of size $q$ and  with socle $T$. 
For twisted groups our notation for $q$ is such that $\PSU_n(q)$ and $\POmega_{2m}^-(q)$ are the twisted groups contained in $\PSL_n(q^2)$ and $\POmega_{2m}^+(q^2)$, respectively.
We write $q=p^e$, for some prime $p$ and some $e\geq 1$, and we define
\begin{equation*}
q_0:= \begin{cases}
q^{2} & \mbox{if $G$ is unitary,}\\
q & \mbox{otherwise.}
\end{cases}
\end{equation*}

We let $V$ be the natural module defined over the field $\mathbb{F}_{q_0}$ of size $q_0$ for the covering group of $T$, and we let $n$ be the dimension of $V$ over $\mathbb{F}_{q_0}$.}
\end{notation}

In studying actions of classical groups, it is rather natural to distinguish between those actions which permute the subspaces of the natural module and those which do not. The stabilizers of subspaces are generally rather large (every parabolic subgroup falls into this class) and therefore the fixed-point-ratio in these cases also tends to be rather large. As the culmination of an important series of papers~\cite{Tim1,Tim2,Tim3,Tim4}, Burness obtained remarkably good upper bounds on the fixed-point-ratio for each finite almost simple classical group in \textit{non-subspace actions}. For future reference, we first need to make the definition of non-subspace action precise.

\begin{definition}\label{def}{\rm Assume Notation~$\ref{generalnotation}$. A subgroup $H$ of $G$ is a \textit{subspace subgroup} if, for each maximal subgroup $M$ of $T$ that contains $H\cap T$,  one of the following conditions hold:
\begin{description}
\item[(a)]$M$ is the stabilizer in $T$ of a proper non-zero subspace $U$ of $V$, where $U$ is totally singular, or non-degenerate, or, if $T$ is orthogonal and $p=2$, a non-singular $1$-subspace ($U$ can be any subspace if $T=\PSL(V)$);
\item[(b)]$M=\O_{2m}^{\pm}(q)$ and $(T,p)=(\Sp_{2m}(q)',2)$.
\end{description}}
\end{definition}
A transitive action of $G$ on a set $\Delta$ is a \textit{subspace action} if the point-stabilizer $G_\delta$ of  $\delta\in \Delta$ is a subspace subgroup of $G$; \textit{non-subspace subgroups} and \textit{actions} are defined accordingly.
For the convenience of the reader we report~\cite[Theorem~$1$]{Tim1}. 
\begin{theorem}[{\cite[Theorem~$1$]{Tim1}}]\label{timthm}
Let $G$ be a finite almost simple classical group acting transitively and faithfully on a set $\Delta$ with point-stabilizer $G_\delta\leq H$, where $H$ is a maximal non-subspace subgroup of $G$. Let $T$ be the socle of $G$. Then, for every $x\in G\setminus\{1\}$, we have \[\mathrm{fpr}_\Omega(x)<|x^G|^{-\frac{1}{2}+\frac{1}{n}+\iota},\]
 where $n$ is defined in~\cite[Definition~2]{Tim1} and either $\iota=0$ or $(T,H,\iota)$ is listed in~\cite[Table~1]{Tim1}.
\end{theorem}
In Theorem~\ref{timthm}, apart from $\PSL_3(2)$ and $\PSp_4(2)'$ (where $n=2$), $n$ is exactly as in Notation~$\ref{generalnotation}$.
The upper bound in Theorem~\ref{timthm} is quite sharp when $n$ is large and is extremely useful for our application. However, for small values of $n$,  Theorem~\ref{timthm} loses all of its power. For instance, when $T=\PSL_4(q)$, we see from~\cite[Table~$1$]{Tim1} that $\iota=1/4$ and hence the upper bound in Theorem~\ref{timthm} only says $\fpr_\Omega(x)\leq |x^G|^0=1$.  However, we point out that  Burness and Guest~\cite{BG} have strengthened Theorem~\ref{timthm} for linear groups of very small rank. 

We are now well equipped for the proof of Theorem~$\ref{proposition:7}$ in the case where $T$ is a non-abelian simple classical group.
 We use the notation that we have established above. We argue by contradiction and we suppose that there exists $g\in G\setminus\{1\}$ with $\fpr_\Omega(g)>1/3$. Without loss of generality, we may assume that $o(g)=2$. Recall that we are proving Theorem~$\ref{proposition:7}$ by induction on $|\Omega|+|G|$. In particular, we may suppose that $G=\langle T,g\rangle$.

\smallskip
\noindent\textsc{Case ``$q\ge 4$".}
\smallskip

\noindent When $q\ge 4$,~\cite[Theorem 1']{LS} yields that the pair $(T,g)$ is in Tables~1 and~2 in \cite{LS}. These tables, together with the pair $(T,g)$, have some additional information on the subgroup $G_\omega$, on a maximal subgroup $M$ of $G$ containing $G_\omega$ and on $\fpr_\Omega(g)$. Using this detailed information, a routine case-by-case analysis yields that $\fpr_\Omega(x)\le 1/3$, for every $x\in G\setminus\{1\}$.  

\smallskip
\noindent\textsc{Case ``$q\le 3$ and $n\le 8$".}
\smallskip

\noindent Since we have only a finite number of cases to check, we have proved the result invoking the help of a computer.
 This was done with a very naive 
algorithm, based on Lemma~\ref{lemma:4}, which we now briefly explain.
 Given a group $X$ and a collection of  subgroup $\mathcal{Y}$ of $X$, we construct all maximal subgroups $\mathcal{M}_Y$ of $Y$, for each $Y\in\mathcal{Y}$. For each $M\in\mathcal{M}_Y$, we construct the permutation representation of $X$ acting on $M\backslash X$. If $X$ contains a non-identity permutation  with $\fpr_{M\backslash X}(x)>1/3$, then we save $M$ in a set $\mathcal{Y}'$, otherwise we disregard $M$ from further analysis. Next, we apply this routine with $\mathcal{Y}$ replaced by $\mathcal{Y}'$. (This computation might seem very time and memory consuming but for most groups this procedure stops after the first iteration.) Finally, we
 consider the collection $\mathcal{X}$ of all subgroups returned in the previous procedure and check that the claim of the theorem holds in each case.

\smallskip

For the rest of the proof, we may assume $q\le 3$ and $n\ge 9$. 

\smallskip
\noindent\textsc{Case ``$G_\omega\le H$, where $H$ is a maximal non-subspace subgroup of $G$".}
\smallskip

\noindent We use the result of Burness described earlier. Then $1/3<|\fpr_{\Omega}(g)|<|g^G|^{-1/2+1/n+\iota}$ and 
\begin{equation}\label{eq:newbad}
|g^G|\le 3^{\frac{1}{2}-\frac{1}{n}-\iota}.
\end{equation} Using the information on $\iota$ in~\cite[Table~1]{Tim1}, some very basic information on the conjugacy classes of $T$ (which can be find in \cite{GLS} or in~\cite{ATLAS} for small groups) and $n\ge 9$, we find with a case-by-case analysis that~\eqref{eq:newbad}  is never satisfied.

\smallskip

For the rest of the proof, we may assume that the only maximal subgroups $H$ of $G$ with $G_\omega\le H$ are subspace subgroups. Case~{\bf (b)} in Definition~\ref{def} does not arise here because $G_\omega$ is a Sylow $2$-subgroup of $G$ (recall~\eqref{sylow}), but $|\mathrm{Sp}_{2m}(q):\mathrm{O}_{2m}^\pm(q)|$ is even. In particular,  every maximal subgroup $H$ of $G$ with $G_\omega\le H$ is in the Aschbacher class $\mathcal{C}_1$.

\smallskip
\noindent\textsc{Case ``There exists a maximal subgroup $H$ of $G$ with $G_\omega\le H$ and ${\bf O}_2(H)\cap g^G=\emptyset$".}
\smallskip

\noindent  Let $\Sigma$ be the system of imprimitivity determined by the overgroup $H$ of $G_\omega$. Consider $\Delta:=\omega^H$ and the permutation group $H^\Delta$ induced by $H$ on $\Delta$. Since $G_\omega$ is a Sylow $2$-subgroup of $G$, we deduce that $G_\omega$ is a Sylow $2$-subgroup of $H$. Therefore, the kernel of the action of $H$ on $\Delta$ is $$\bigcap_{h\in H}G_\omega^h={\bf O}_2(H).$$
Thus $H^{\Delta}\cong H/{\bf O}_2(H)$ and hence ${\O }_2(H^\Delta)=1$.
We have
\begin{equation}\label{equation:equation}\Fix_{\Omega}(g)=\bigcup_{\Delta'\in\Fix_{\Sigma}(g)}\Fix_{\Delta'}(g).\end{equation}
Let $\Delta'\in \Fix_\Sigma(g)$. Then $\Delta'=\Delta^x$, for some $x\in G$, and $H^x$ is the setwise stabilizer of $\Delta'$ in $G$. Therefore ${\bf O}_2((H^x)^{\Delta'})=1$. Since $|\Delta'|<|\Omega|$, by induction we have that either $\fpr_{\Delta'}(g)\le 1/3$, or $g$ fixes pointwise each element of $\Delta'$. However the latter case does not arise because the kernel of the action of $H^x$ on $\Delta'$ is ${\bf O}_2(H)^x$ and we are assuming ${\bf O}_2(H)\cap g^G=\emptyset$. Thus $g$ does not act trivially on $\Delta'$ and $\fpr_{\Delta'}(g)\le 1/3$. Now,~\eqref{equation:equation} yields $\fpr_\Omega(g)\le 1/3$.

\smallskip

For the rest of the proof, we may suppose that, every maximal subgroup $H$ of $G$ with $G_\omega\le H$, is in the Aschbacher class $\mathcal{C}_1$ and satisfies ${\bf O}_2(H)\cap g^G\ne \emptyset$. In particular, ${\bf O}_2(H)\ne 1$ and hence $H$ is a $2$-local subgroup. Using the information on the maximal subgroups $H$ of the finite classical groups in the Aschbacher class $\mathcal{C}_1$ in~\cite[Section~$4.1$]{KL} and the fact that ${\bf O}_2(H)\cap g^G\ne\emptyset$, we obtain $g\in T$ and $$G=\langle T,g\rangle=T.$$ 

\smallskip
\noindent\textsc{Case ``$q=3$''.}
\smallskip

\noindent Using again the information in~\cite[Section~$4.1$]{KL}, we have that, if $H$ is a maximal subspace subgroup with $G_\omega\le H$ and $q=3$, then $|{\bf O}_2(H)|\le 2$. In particular, we are in the position to refine slightly the argument in~\eqref{equation:equation}.  Let $\Sigma$ be the system of imprimitivity determined by the overgroup $H$ of $G_\omega$. We partition $\Fix_{\Sigma}(g):=\Sigma_1\cup\Sigma_2$ in two subsets: $\Sigma_1$ consists of the $\Delta\in \Fix_{\Sigma}(g)$ with $g$ not fixing pointwise $\Delta$ and $\Sigma_2$ consists of the $\Delta\in \Fix_{\Sigma}(g)$ with $g$ fixing pointwise $\Delta$. Observe that, if $\Delta\in \Sigma_2$, then $g\in {\bf O}_2(G_{\{\Delta\}})$ and hence ${\bf O}_2(G_{\{\Delta\}})=\langle g\rangle$. If $\Sigma_2$ contains two distinct elements $\Delta$ and $\Delta'$, we deduce ${\bf O}_2(G_{\{\Delta\}})=\langle g\rangle={\bf O}_2(G_{\{\Delta'\}})$ and hence $\langle g\rangle$ is centralized by $\langle G_{\{\Delta\}},G_{\{\Delta'\}}\rangle=G$, where the last equality follows from the maximality of $G_{\{\Delta\}}$ and $G_{\{\Delta'\}}$ in $G$ and from $\Delta\ne \Delta'$. Therefore $|\Sigma_2|\le 1$. Thus, applying the inductive hypothesis for the action of $g$ on $\Delta\in \Fix_{\Sigma}(g)$, from~\eqref{equation:equation} we deduce
$$|\Fix_\Omega(g)|=\sum_{\Delta\in\Sigma_1}|\Fix_{\Delta}(g)|+\sum_{\Delta\in\Sigma_2}|\Fix_\Delta(g)|\le |\Sigma_1|\frac{|\omega^H|}{3}+|\Sigma_2||\omega^H|\le |\omega^H|\left(\frac{|\Fix_\Sigma(g)|-1}{3}+1\right).$$
Since $g\ne 1$, $g$ does not act trivially  on $\Sigma$ and hence $|\Fix_\Sigma(g)|\le |\Sigma|-2$. Thus
$$
|\Fix_\Omega(g)|\le |\omega^H|\left(\frac{|\Sigma|-3}{3}+1\right)=\frac{|\omega^H||\Sigma|}{3}=\frac{|\Omega|}{3}.
$$

\smallskip
\noindent\textsc{Case ``$q=2$''.}
\smallskip

\noindent  Here, $G=\langle T,g\rangle$ is one of the following groups $\mathrm{PSL}_n(2)$, $\mathrm{PSU}_n(2)$, $\mathrm{PSp}_n(2)$, $\mathrm{P}\Omega_n^+(2)$ and $\mathrm{P}\Omega_n^-(2)$.

Recall that $V$ is the underlying module for $G$. Let $W$ be a  totally isotropic subspace of $V$ of dimension $1$ fixed by $G_\omega$ and let $H$ be the stabilizer of $W$. (The existence of $W$ is guaranteed by the fact that $G_\omega$ is a $2$-group and by the fact that $V$ has characteristic $2$.) Now, $H$ is a maximal  $\mathcal{C}_1$-subgroup of $G$ with $G_\omega \le H$. Considering the cases we are left at this point of the proof, $H$ is a maximal parabolic subgroup of $G$. Since ${\bf O}_2(H)\cap g^G\ne \emptyset$, we deduce from the structure of $H$ in~\cite{KL}, that $g$ is a transvection, that is, $\dim\cent V g=n-1$.

Let $U$ be the group of upper unitriangular matrices in $\mathrm{GL}_n(2)$ and  let $T(n)$ be the number of transvections in $U$. It is easy to show arguing inductively on $n$ that $T(n)=2T(n-1)+2^{n-1}-1$. Using this recursive relation and the fact that $T(2)=1$, we obtain $T(n)=(n-2)2^{n-1}+1$.

From the previous paragraph, we have $|g^G\cap G_\omega|\le (n-2)2^{n-1}+1$ and hence
$$\fpr_\Omega(g)=\frac{|g^G\cap G_\omega|}{|g^G|}\le \frac{(n-2)2^{n-1}+1}{|G:\cent Gg|}\le \frac{1}{3},$$
where the last inequality follows by comparing $|G:\cent G g|$ with $(n-2)2^{n-1}+1$. (Information on $|G:\cent G g|$ can be found for all the groups under consideration either in~\cite{KL} or in~\cite{GLS}.)

This completes the proof of Theorem~\ref{proposition:7}.

\end{document}